\documentclass[12pt,leqno]{amsart}
\usepackage{xcolor}
\usepackage{euscript, amssymb, amsmath, amsthm}
\usepackage{epsfig}
\usepackage{graphicx}
\usepackage{caption}
\usepackage{longtable}
\usepackage{dcolumn}
\usepackage{setspace}
\usepackage[most]{tcolorbox}
\definecolor{webred}{rgb}{0.75,0,0}
\definecolor{webgreen}{rgb}{0,0.75,0}
\definecolor{refkey}{gray}{0.75}
\usepackage[pagebackref=true, colorlinks=true, citecolor=blue]{hyperref}

\numberwithin{equation}{section}

\def\R{\mathbb{R}}

\def\d{\displaystyle}
\def\e{{\varepsilon}}

\def\intspc{\int_{\R^N}}

\def\til{~}

\DeclareMathOperator*{\supp}{supp}

\usepackage[color]{showkeys}
\newtheorem{thm}{Theorem}%[section]
\newtheorem{cor}{Corollary}%[section]
\newtheorem{lem}{Lemma}[section]
\newtheorem{rem}{Remark}[section]
\newtheorem{Def}{Definition}%[section]

\date{}

\subjclass[2010]{35L71,  35B44}
\keywords{blow-up, lifespan, nonlinear wave equations, scale-invariant damping, time-derivative nonlinearity.}

\tcbset{
    frame code={}
    center title,
    left=0pt,
    right=0pt,
    top=0pt,
    bottom=0pt,
    colback=gray!10,
    colframe=white,
    width=\dimexpr\textwidth\relax,
    enlarge left by=0mm,
    boxsep=5pt,
    arc=0pt,outer arc=0pt,
    }

\begin{document}

\title{ Blow-up and lifespan estimate for   wave equations
 with critical damping term of space-dependent type
 related to Glassey conjecture}

\author[Ahmad Z. Fino  and M. A. Hamza]{Ahmad Z. Fino$^{1}$ and Mohamed Ali Hamza$^{2}$}
\address{$^{1}$ Department of Mathematics, Sultan Qaboos University
 FracDiff Research Group (DR/RG/03),  P.O. Box 46, Al-Khoud 123, Muscat, Oman.}

\address{$^{2}$ Department of Basic Sciences, Deanship of Preparatory Year and Supporting Studies, Imam Abdulrahman Bin Faisal University, P.O. Box 1982, Dammam 34212, SAUDI ARABIA.}

\medskip

\email{a.fino@squ.edu.om; ahmad.fino01@gmail.com (Ahmad Z. Fino)}
\email{mahamza@iau.edu.sa (M.A. Hamza)}

\pagestyle{plain}

%%%%%%%%%%%%%%%%%%%%%%%%%%%%
%%%%%%%%%%%%%%%%%%%%%%%%%%%%

\maketitle

\begin{abstract}
The main purpose of the present paper is to study the blow-up problem
of the  wave equation with  space-dependent damping in the \textit{scale-invariant case} and  time derivative  nonlinearity with small initial data.
Under  appropriate  initial data which are compactly  supported,
 by using a test function method and taking into account the effect of the  damping term
 ($\frac{\mu}{\sqrt{1+|x|^2}}u_t$),  we provide
 that  in higher dimensions the  blow-up region  is given by $p \in (1, p_G(N+\mu)]$ where $p_G(N)$ is the Glassey exponent.
%We established also the lifespan estimate
% on the blow-up region  of a
% The blow-up region  obtained is
%  similar to that derived in the case of the time-dependent type damping in \cite{Our2}.
Furthermore, we shall     establish  a
blow-up region,  independent of $\mu$ given by
$p\in  (1, 1+\frac{2}{N}),$ for appropriate initial data
 in the
energy space with  noncompact support.
%In particular,
%the  blowing-up solutions even for
%$p > p_G(N+\mu)$, in the case $\mu>1$

% the one space dimensional case, we a blow-up result,     with some initial data, for  any  $p\in (1,3)$.
%{\color{red}(Please delete it) Finally,
%to  the corresponding  equation with  time-dependent damping ($\frac{\mu}{{1+t}}u_t$).
%We point out that the above  result    is an improvement of  the results obtained in \cite{Our2} if $\mu> N-1$.}
\end{abstract}

% The blow up results and the upper bound of lifespan
%estimates are obtained by the so-called test function method. The key ingre-
%dient is to construct special positive solutions to the linear dual problem with
%the desired asymptotic behavior, which is reduced, in turn, to constructing
%solutions to certain elliptic \eigenvalue" problems.

\date{\today}

% 35L71 Semilinear second-order hyperbolic equations
% 35B44  Blow-up

\pagestyle{plain}

\medskip

\noindent {\bf MSC 2010 Classification}:  35B44,  35L71, 35L15.

\noindent {\bf Keywords:}   blow-up, lifespan, nonlinear wave equations, scale-invariant damping, time-derivative nonlinearity.

%%%%%%%%%%%%%%%%%%%%%%%%%%%%%%%%%%%%%%%%%%%%%%%%%%
%%%%%%%%%%%%%%%%%%%%% SECTION1 %%%%%%%%%%%%%%%%%%%%%%%
%%%%%%%%%%%%%%%%%%%%%%%%%%%%%%%%%%%%%%%%%%%%%%%%%%

\section{Introduction}
\par\quad
In this work, we consider the  semilinear  wave equations with a power-nonlinearity of derivative type
\begin{equation}
\label{A}
\left\{
\begin{array}{ll}
\partial_{t}^2u - \ \Delta u +V(x)\partial_tu= |\partial_tu|^p &\quad\text{in $[0, T)\times\R^N$}, \\
u(x,0)=\e f(x), \quad \partial_tu(x,0)=\e g(x) &\quad x\in\R^N,
\end{array}
\right.
\end{equation}
where $V(x)=\frac{\mu}{\sqrt{1+|x|^2}}$,  $N\ge1$, $p>1$, and  $\mu\ge0$.
 Moreover, the parameter $\e$ is a positive number describing the size of the initial
data.\\ % and $f\in H^1(\R^N)$ and $g\in L^2(\R^N)$ are positive functions which are compactly supported on  $B_{\R^N}(0,R), R>0$.

\par

It is worth-mentioning that the presence of nonhomogeneous  damping term in \eqref{A} has an interesting
impact on the (global) existence or the nonexistence of the solution of \eqref{A} and
its lifespan. Hence, it is natural to study the influence of the nonlinear terms on the
behavior of the solution and see whether or not this may produce a kind of competition
between the damping term  $V(x) \partial_tu$ and the time derivative  nonlinearity $|\partial_tu|^p$.\\

As noted in \cite{I}, equation \eqref{A} can modelize the
wave travel in a nonhomogeneous gas with damping, and the space dependent
coefficients represent the friction coefficients or potential. \\
% Obviously,   when
%we consider the  nonlinearity of the form  $|u|^p$ or the more general nonlinearity of the form $F(u,\partial_tu)$
%can be useful to medelise this  kind of physical problems.
%
%%
%
%The corresponding linear equation  to \eqref{A} is given by
%\begin{equation}\label{AL}
%\d \partial^2_{t}u^L-\Delta u^L+V(x)\partial_tu^L=0.
%\end{equation}
% ??????????????????? could you add here  one paragraph
%related to the linear problem

 \par

 The semilinear wave equation equation for classical semilinear wave equation   with  power nonlinearity namely
 %and with  power nonlinearity  has been studied extensively,
  %It is interesting to recall the known result  for classical semilinear wave equation  with  power nonlinearity c
\begin{equation}
\label{P0}
\left\{
\begin{array}{ll}
 \partial_{t}^2u - \ \Delta u = |u|^p &\quad\text{in $[0, T)\times\R^N$}, \\
 u(x,0)=\e f(x), \quad \partial_tu(x,0)=\e g(x) &\quad x\in\R^N.
\end{array}
\right.
\end{equation}
has been studied extensively.
For small initial data, non negative   and compactly  supported, the  critical power   is so-called   the Strauss exponent  is   denoted by $p_S$  and is
 given  by $p_S(N)=\frac{N+1+\sqrt{N^2+10N-7}}{2(N-1)}.$
Indeed,  if $p \le p_S$ then there is  a blow-up solution for  \eqref{P0} %with small initial data  compactly  supported,
 and for $p > p_S$ a global solution exists; see e.g. \cite{John3,Strauss,YZ06,Zhou} among many other references.\\

 \par

 Coming back to \eqref{A} and
  by ignoring the damping term
 $V(x)\partial_tu$, the  problem
is reduced to the classical semilinear wave equation  namely
\begin{equation}
\label{P}
\left\{
\begin{array}{ll}
 \partial_{t}^2u - \ \Delta u = |\partial_t u|^p &\quad\text{in $[0, T)\times\R^N$}, \\
 u(x,0)=\e f(x), \quad \partial_tu(x,0)=\e g(x) &\quad x\in\R^N,
\end{array}
\right.
\end{equation}
 for which we have the Glassey conjecture. This case  is characterized by a critical power, denoted by $p_G$,  and  given by
\begin{equation}\label{Glassey}
p_G=p_G(N):=1+\frac{2}{N-1}.
\end{equation}
More precisely,  if $p \le p_G$ then there is no global solution for  \eqref{P},
for small initial data, non negative   and compactly  supported, 
 %under suitable sign assumptions for the initial data,
  and for $p> p_G$ a global solution exists for small initial data; see e.g. \cite{Hidano1,Hidano2,John1,Rammaha,Sideris,Tzvetkov,Zhou1}.\\

%%%%%%%%%%%%%%%%%%
%%%%% section 2 %%%%%%%
%%%%%%%%%%%%%%%%%%

 In the case where   the damping term is given by $\frac{\mu}{t+1}\partial_tu$ instead of
 $V(x)\partial_tu$, then the equation \eqref{A}
becomes
\begin{equation}
\label{B}
\left\{
\begin{array}{ll}
\d \partial^2_{t}u-\Delta u+\frac{\mu}{1+t}\partial_tu=|\partial_tu|^p,
&\quad \mbox{in}\ \R^N\times[0,\infty),\\
u(x,0)=\e f(x),\ \partial_tu(x,0)=\e g(x), &\quad  x\in\R^N.
\end{array}
\right.
\end{equation}
Concerning the blow-up results and lifespan estimate of the solution of \eqref{B}, a first
 blow-up region was obtained in
 \cite{LT2}
for $p \in (1, p_G(N+2\mu)]$.
Later, an important refinement  was performed in
 \cite{Palmieri}, using the
integral representation formula, where the new bound is given by
 $p \in (1, p_G(N+\mu)]$, for $\mu \in (0,2)$. Recently,
thanks to a better understanding of the corresponding linear problem to (\ref{B}),  an improvement in \cite{Our2}
shows that $p \in (1, p_G(N+\mu)]$ is  probably  the new critical exponent  for $\mu >0$. \\

\par

Turning to the analogous  nonlinear problem \eqref{A} with
${\mu}/{\sqrt{1+|x|^2}}$ being  changed by
${\mu}/{{(1+|x|^2)}^{\beta/2}},$ for some  $\beta>1$.
In this case, it is  reasonable to  expect that
the blow up region is similar to the case of pure   wave equation and  the scattering damping term has no
influence  in the dynamics.
The predicted blow up result   was  obtained in the  case $\beta>2$,  by Lai and Tu  \cite{LT}.
However, up to our knowledge, there is no result in the case $\beta \in (1,2].$\\

\par

The emphasis in the first part of the   manuscript  is to establish a blow-up results for  solution of
%to the study the blow-up of the solutions of the Cauchy problem
\eqref{A}, in the case where the initial data has compact support,
and to determine a candidate
as critical exponent.
Clearly, in the scale-invariant case  the situation is
different.
 Therefore, the goal  will be on the comprehension
of the influence of the damping term $ V(x)\partial_tu$ on the blow-up result and the lifespan estimate. In fact,
our
target is to give the upper bound, denoted here by $p_G(N,\mu)$, delimiting a
blow-up region for the energy solution of equation \eqref{A}.
First, as observed for  the problem \eqref{B}, where the damping produces a shift in $p_G$ in
the dimensional parameter of magnitude  $\mu$, we expect that the same phenomenon
holds for \eqref{A}. In other words, we predict that the upper bound is given by:
\begin{equation}
p_G(N,\mu):= p_G(N+\mu)=1+\frac{2}{N+\mu-1}, \quad \textrm{if} \quad N\ge 1.
\end{equation}

%
%\begin{equation}\label{pG}
%p_G(N,\mu):=
%\d \left\{
%\begin{array}{ll}
%  p_G(N+\mu)=1+\frac{2}{N+\mu-1}
% &
% \ \text{if} \
% N\ge 2, \vspace{.1cm}
% \\
% \infty
%&
% \ \text{if} \  N=1,
%\end{array}
%\right.
%\end{equation}
%
%\par

The argument which led to our
 blow up result obtained here in the case for some   initial data where the support is  compact,
  is
by employing the test function method.  In fact,  we shall use   a test function as product of a cut-off function and  an explicit solution of the conjugate equation corresponding to  the linear problem of \eqref{A}.   Let us denote that this strategy is
  inspired by  \cite{ISW, LS,LT3}. \\

\par

In the second part, we consider   the solution of \eqref{A} in  the case where  the initial data   is   decaying slowly at infinity.
%  The context  is  a little bit  different.

 %First,% on  the comparison  of the case where the support of the initial data   is compact
   %context  is  a little bit  different. Indeed,
  %in  
%  Before considering our 
 % The analogue   problem of 
 
 It is well-known that  the solution of  \eqref{P0} 
   %namelythe semilinear wave equation with source term $|u|^p$,
  blow up  for any $p>1$, for suitable   initial data     decaying slowly at infinity. 
  Indeed,
 if 
 $(f,g)$  satisfies  that
\begin{equation*}
  f(x)\equiv 0 \  \ \textrm{and}\ \  g(x)\ge \frac{\Pi_0 (x)}{(1+|x|)^{1+\kappa}},
  \end{equation*}
where $\Pi_0\equiv C$, if $0<\kappa<\kappa_0:=\frac{2}{p-1}$, and $\Pi$ is positive, monotonously increasing, $\d{\lim_{r\rightarrow\infty}\Pi_0(r)=\infty}$ if $\kappa=\kappa_0$, 
 then 
the system  \eqref{P0} has  a blow up solution for any $\kappa\in (0,\kappa_0]$. 
% when the initial data  $f$ and $g$ are  decaying slowly at infinity for any $p>1$. 
%Consequently, if also  the  initial data  is in the energy space, we provide  a
% blow up  result  even for some  supercritical range ($p>p_S$). 
We mention  
 the pioneering results on non-compactly supported initial
data by Asakura \cite{A}  and also \cite{T00,T0}.
Furthermore, in the supercrtical range  $p>p_S(N)$, 
 %for the existence of solutions  , In \cite{Hi} the autor showed   
  a global (in time)  result for  solution  to \eqref{P0} in \cite{Hi},
   for small initial data      in the weighted space
 %$r^{\kappa}(|f'(r)|+|g(r)|)$ is small in 
 $L^{\infty}(\R^N,(1+r^{\kappa}))$, if $\kappa>\kappa_0$.
 Therefore,   the   critical (in the sense of interface between blow-up and global existence in the case $p>p_S(N)$)
decay of the initial data
is  $\kappa_0$.
 Let us denote that 
  if $u$ is a solution of \eqref{P0},
then for all $\lambda>0$,  $u_{\lambda}(x,t)=\lambda^{\kappa_0}u(\lambda x, \lambda t)$ is also a solution.
Therefore, the critical value of $\kappa_0$ is somehow related to the scaling of the equation \eqref{P0}. \\

%Due to invariance  under the above  scaling  and the   results,   
%Therefore, if $p>p_G(N)$,  the expected  critical (in the sense of interface between blow-up and global existence)
%decay of the initial data
%is probably $\kappa_0=\frac1{p-1}$.\\

%Furthermore,  global existence for the solution of \eqref{P0} with small data in the supercritical case has been proved only for $\kappa>\frac2{p-1}$; see e.g. cite . Therfore, the critical exponent for $\kappa_0$   \\

In the same  way, %  it is interesting  to recall the known  blow up result  for  a weak   solution of \eqref{P}.
 %The study of the existence or nonexistence of solutions to \eqref{P} has been widely studied.
 it is proven, in \cite{K}, when $N=2$ and $N=3$  that a   solution of \eqref{P}  blows up in finite time, in the case where the initial data satisfies:
  \begin{equation*}
  f(x)\equiv 0 \  \ \textrm{and}\ \  g(x)\ge \frac{M}{(1+|x|)^{\kappa}},
  \end{equation*}
  for some positive constant $M$ and 
  $0<\kappa<\kappa_1:=\frac{1}{p-1}$.
  Moreover,  the  lifespan  $T_\e$
%of the maximal existence time 
satisfies
$T_\e \leq\ \d
 C \, \e^{-\frac{p-1}{1-\kappa (p-1)}}$.  This result was improved in \cite{T}, for any $N\ge2$.
 (see  also \cite{W} for more general initial data). 
 On the other hand, in the case $N=3$, %for the existence of solutions  , In \cite{Hi} the autor showed   
  a global  result for  solution  to \eqref{P} in \cite{Hi},
   for radial initial data   small   in the weighted space
 %$r^{\kappa}(|f'(r)|+|g(r)|)$ is small in 
 $L^{\infty}(\R^N,(1+r^{\kappa}))$, if $\kappa>\kappa_1$ and  $p>p_G(N)$.
 In addition, we remark that if $u$ is a solution of \eqref{P},
then for all $\lambda>0$,  $u_{\lambda}(x,t)=\lambda^{\kappa_1}u(\lambda x, \lambda  t)$ is also a solution.
%Due to invariance  under the above  scaling  and the   results,   
Therefore,   the expected  critical %(in the sense of interface between blow-up and global existence  if $p>p_G(N)$)
decay of the initial data
is $\kappa_1$.\\
 
%
%
% According to Asakura's observation, there exists a critical decay of  a suitable initial data
%such that the  solution of \eqref{P}  blow-up even for $p>p_S(N)$  for an  appropriate initial data
%with slow spatial decay.
%
%

As we said before, we are interested in this part in studying the  blow-up result of the solution of \eqref{A} in  the case where the support of the initial data   is  decaying slowly at infinity. 
%More precisely, $f$ and $g$ satisfies
%\begin{equation}
%  f\in L^1(\R^N) \  \ \textrm{and}\ \
%\lim_{|x|\to \infty}|x|^{\kappa}g(x)= L,
%  \end{equation}
%For some $L>0$.
%Furthermore,  for
 %suitable conditions on
%initial data where the support is  {\bf non compact},
By applying  the test function method for  a cut-off function, we  derive a  blow up result  for weak solution of \eqref{A}. In particular, %by chossing %$\kappa>\frac{N}2$,    
%the initial data   in  the energy space,
we deduce   
a blow up result  for  an energy solution to \eqref{A},
if $p\in (1,  p_0(N))$, where
\begin{equation}
p_0(N):=1+\frac{2}{N}, \quad \textrm{for all} \quad N\ge 1.
\end{equation}
% for square integrable initial data,
 %the   energy solution of equation \eqref{A}.\\
%no global solutions of  \eqref{A}  with some initial data.
%{\color{red}(Please delete it) Finally,  we   get a similar  blow-up  result for the solution of equation \eqref{B}.  We point out that this new result stated in Theorem 3  enhances the result obtained in \cite{Our2}.}\\

\par

This paper is organized as follows:
 First,  Section 2 is devoted to the definition  of the weak formulation of  \eqref{A}, in the energy space and the definition of weak solution, together with the
statement of the main theorems of our work.
 Then, in Section 3,  we get  a blow-up result in higher dimensions as stated in Theorem 1.
%{\color{red}(Please reformulate this sentence)
 Finally,
in Section 4,  we  establish  a new  blow up result for weak   solutions of the problem  \eqref{A}   with some initial data  as stated in Theorem 2.
\section{Main Result}\label{sec-main}
\par
This section is devoted to the statement of the main results. However, before that we
start by  giving the definition of energy solution for our problem \eqref{A}.
\begin{Def}\label{def1}
Let $N\ge1$, $f\in H^1(\mathbb{R}^N)$, $g\in L^2(\mathbb{R}^N)$  and $T>0$.  Let $u$ be such that
$u\in \mathcal{C}([0,T),H^1(\R^N))\cap \mathcal{C}^1([0,T),L^2(\R^N)) \
\text{and} \ \partial_tu \in L^p_{loc}((0,T)\times \R^N),$
verifies,  for any $\varphi \in \mathcal{C}^1_0\left([0, T)\times \R^N\right) \cap \mathcal{C}^{\infty}\left((0, T)\times \R^N\right)$, the following identity:
\begin{equation}\label{weaksol}
\begin{split}
& \e \int_{\R^N}g(x)\varphi( x,0)dx
+\int_0^T\int_{\R^N}|\partial_tu|^p\varphi(x, t) \, dxdt \\
=& \int_0^T\int_{\R^N} - \partial_tu(t, x)\partial_t\varphi( x,t) \, dxdt
+ \int_0^T\int_{\R^N}
\ \nabla u(t, x)\cdot\nabla\varphi(x, t) \, dxdt \\
&+ \int_0^T\int_{\R^N} V(x) \partial_tu(t, x)\varphi( x,t) \, dxdt
\end{split}
\end{equation}
and the condition $u(x,0)=\varepsilon f(x)$ is satisfied. Then, $u$ is called an {\bf energy solution} of
 (\ref{A}) on $[0,T)$.

 We denote the lifespan for the energy solution by:
$$T_{\e}(f,g):=\sup\{T\in(0,\infty];\,\,\hbox{there exists a unique energy solution $u$ of \eqref{A}}\}.$$
Moreover, if $T>0$ can be arbitrary chosen, i.e. $T_{\e}(f,g)=\infty$, then $u$ is called a global energy solution of \eqref{A}.
\end{Def}
 Furthermore, we shall
write   the
definition of weak solutions for the  problem \eqref{A}.
\begin{Def}\label{def2}
Let $N\ge1$,  $f\in L^1_{loc}(\R^N), g\in L^1_{loc}(\R^N)$ and $T>0$.
 Let $u$ be such that
$ u \in L^1_{loc}((0,T)\times \R^N)
 \text{and} \ \partial_tu \in L^p_{loc}((0,T)\times \R^N),$
verifies, for all   $\varphi \in \mathcal{C}^1_c\left([0, T)\times \R^N \right) \cap \mathcal{C}^{\infty}\left((0, T)\times \R^N
\right)$, the following:
\begin{equation}\label{weaksolbis}
\begin{split}
& \e \int_{\R^N}g(x)\varphi( x,0)dx
+\int_0^T\int_{\R^N}|\partial_tu|^p\varphi(x, t) \, dxdt \\
=& \int_0^T\int_{\R^N} - \partial_tu(t, x)\partial_t\varphi( x,t) \, dxdt
- \int_0^T\int_{\R^N}
u(t, x)\Delta \varphi(x, t) \, dxdt \\
&+ \int_0^T\int_{\R^N} V(x) \partial_tu(t, x)\varphi( x,t) \, dxdt
\end{split}
\end{equation}
and the condition $u(x,0)=\varepsilon f(x)$ is fulfilled. Then, $u$ is called a {\bf weak  solution} of
 (\ref{A})  on $[0,T)$.

  We denote the lifespan for the weak solution by:
$$T_w(f,g):=\sup\{T\in(0,\infty];\,\,\hbox{there exists a unique weak solution $u$ of \eqref{A}}\}.$$
Moreover, if $T>0$ can be arbitrary chosen, i.e. $T_w(f,g)=\infty$, then $u$ is called a global weak solution of \eqref{A}.
\end{Def}

Let us mention that, by integrating by parts, an energy solution to \eqref{A} is also  a weak solution to \eqref{A}.
%
%Obviously, the  formulation corresponding to  an energy solution  of  \eqref{B}, repectively, weak solution of  \eqref{B}  can be also given by   \eqref{weaksol}, respectively, \eqref{weaksolbis}
%simply changing the expression of $V(x)$  by $V(t)$ in
% the last term in the right-hand side of  \eqref{weaksol}, respectively, \eqref{weaksolbis}.\\

\par

The following theorems state the main results of this article.

\par

\begin{thm}
\label{th1}
Let $R>0$, $N\ge 2$, $\mu\ge0$ and
$1<p\le p_G(N+\mu).
$
Assume that $f\in H^1(\R^N)$, $g\in L^{2}(\R^N)$
are compactly supported functions   on  $B_{\R^N}(0,R)$  %\textcolor{red}{which
%are non-negative and
  %do not vanish everywhere (why we need such conditions)} 
  and satisfy
  \begin{equation}\label{C0}
   \intspc \big(\Delta f(x)+ g(x)\big)\phi(x) dx >0.
  \end{equation}
  where $\phi(x)$  is  a solution of the elliptic problem \eqref{phi}.\\
Suppose that $u$ is an energy solution of \eqref{A} with compact support
\begin{equation}\label{suppcond}
\supp u \in	\left\{(x,t) \in \R^N   \times   [0,T) \colon |x| \le R+ t\right\}.
\end{equation}
Then, there exists a constant $\e_0=\e_0(f, g, \mu, N, p,R)>0$
such that the lifespan $T_\e$ verifies
\begin{equation}\label{T-epss}
T_\e \leq
\d \left\{
\begin{array}{ll}
 C \, \e^{-\frac{2(p-1)}{2-(N+\mu-1)(p-1)}}
 &
 \ \text{for} \
 1<p<p_G(N+\mu), \vspace{.1cm}
 \\
 \exp\left(C\e^{-(p-1)}\right)
&
 \ \text{for} \ p=p_G(N+\mu),
\end{array}
\right.
\end{equation}
for $0<\e\le\e_0$ and some constant $C$ independent of $\e$.
\end{thm}

To state our second result, we define,  
 $\Pi_1\equiv C_1$, if $0<\kappa<\frac{1}{p-1}$, and $\Pi_1$ is positive, monotonously increasing, $\d{\lim_{r\rightarrow\infty}\Pi_1(r)=\infty},$ if $\kappa=\frac1{p-1}$. \\

Here is the statement of second   theorem in this paper:
\begin{thm}
\label{th2}
Let  $N\geq1$, $\mu\ge0$ and
$p>1$. Assume that $f\in  L^1_{loc}(\R^N)$  and $g\in L^1_{loc}(\R^N)$.
%are a non-negative functions
%such that
  %\begin{equation}\label{C00}
%\lim_{|x|\to \infty}|x|^{\kappa}g(x)= L%_{\pm}.
  %\end{equation}
  %for some $L>0$ and    $\kappa<\frac1{p-1}$. 
  Suppose that $u$ is a weak solution of \eqref{A}. Therefore, \\

i)
If  $f\equiv0$ and $g$ satisfies
\begin{equation}\label{case1t}
g(x)\geq\frac{\Pi_1(|x|)}{(1+|x|)^\kappa},\end{equation}
for some $\kappa\le\frac1{p-1}$, then the  solution of \eqref{A} blows-up in finite  time. Moreover, if $\kappa<\frac1{p-1}$, 
 there exists a constant $\e_0=\e_0(g, \mu,  p)>0$
such that the lifespan $T_w$ verifies
\begin{equation}\label{T-eps10av}
T_w \leq\
\d 
 C \, \e^{-\frac{p-1}{1-\kappa (p-1)}}, \qquad \forall \e \in ( 0,\e_0].
\end{equation}

ii) If $g\ge 0$ and $f$ satisfies 
\begin{equation*}
\Delta f(x)\geq\frac{\Pi_1(|x|)}{(1+|x|)^{\kappa+1}},
\end{equation*} 
for some $\kappa\le \frac1{p-1}$, then the  solution of \eqref{A} blows-up in finite  time. Moreover, Moreover, if $\kappa<\frac1{p-1}$, 
 there exists a constant $\e_0=\e_0(f,g, \mu,  p)>0$
such that the lifespan $T_w$ verifies
\eqref{T-eps10av}.\\

  iii)
  If  $f\in L^1(\R^N)$ and $g$ satisfies \eqref{case1t}, 
for some    $\kappa<\min (N+1,\frac1{p-1})$,  then the  solution of \eqref{A} blows-up in finite  time.
 Moreover,
 there exists a constant $\e_0=\e_0(f,g, \mu,  p)>0$
such that the lifespan $T_w$ verifies
\eqref{T-eps10av}.\\
\end{thm}

By exploiting the fact that, if $f\in H^1(\R^N)$ and $g\in L^2(\R^N)$
and  by integrating by parts, we conclude  that   a weak solution of \eqref{A}  with initial data in $H^1(\R^N)\times L^2(\R^N)$ is also  an energy solution of \eqref{A}.
Therefore,
an important consequence of  Theorem \ref{th2}
is the following result:
\begin{cor}
\label{cor1}
Let  $\mu\ge0$ and
$1<p<1+\frac{2}{N}.$
Assume that $f\in  H^1(\R^N)\cap L^1(\R^N)$  and $g\in L^2(\R^N) $
and $g$ satisfiying \eqref{case1t}, 
for some     $ \kappa\in (\frac{N}2,\frac1{p-1})$ and   $\kappa<N+1$,  then 
the weak solution  $u$ of \eqref{A} is an energy  solution. Therefore the  solution of \eqref{A} blows-up in finite  time. Moreover,
 there exists a constant $\e_0=\e_0(f,g, \mu,  p)>0$
such that the lifespan $T_\e$ verifies
\begin{equation}\label{T-eps}
T_\e \leq\
\d
 C \, \e^{-\frac{p-1}{1-\kappa (p-1)}},
\end{equation}
for $0<\e\le\e_0$ and some constant $C$ independent of $\e$.\\

%In addition,
% there exists a constant $\e_0=\e_0(f,g, \mu,  p)>0$
%such that the lifespan $T_\e$ verifies
%
%\begin{align}\label{TT}
% T_{\e}\le C\max\Big(
% \e^{-\frac{p-1}
% {1-\kappa(p-1)}},
% \e^{-\frac1{N+2-p^{\prime}}}
% \Big).
%\end{align}
%for $0<\e\le\e_0$ and some constant $C$ independent of $\e$.
\end{cor}

\par

\begin{rem}
The result obtained in Theorem \ref{th1} holds only in high space dimensions, $N\ge2$.
% We avoid the case $N=1$ in view of the fact that the function $\phi$ defined  in \eqref{phi} below satisfies  is defined only if $N\ge 2$ (see \cite{LLTW}).
 Although, we expect the result is  also true in the one dimentional case, this has to be confirmed.
\end{rem}

\begin{rem}
%%%%%%%%%%%%%%%%%%%%%%%%%%%%%%%%%%
 The blow up result  stated in Corollary \ref{cor1}
   shows that there exist a   blow-up region not depending  on the  parameter $\mu$. In addition,
 we deduce that, in the case $\mu>1$,    the result obtained in Corollary \ref{cor1} implies that, for some
    initial data in the energy space  with non compact support,
   there is a new region %where we have  blow up solution
%% with critical damping term of time-dependent type \eqref{T-sys1}
 comparing  to the result  obtained in Theorem \ref{th1}.
%   the blow-up interval, $p \in (1, p_G(1+\mu)]$ obtained recently  in \cite{Our2}.
\end{rem}

%
%\begin{rem} By combining the results derived in Therem \ref{th1},    Corollary \ref{cor1}  and  employing the fact that
% $$ p_G(N,\mu)<p_0(N)  \qquad  \textrm{if} \qquad \mu>1,$$ we deduce that
%\begin{equation}\label{pc}
%p_c(N,\mu) \geq
%\d \left\{
%\begin{array}{ll}
% p_0(N)=1+\frac2{N}
% &
% \ \text{for} \
% \mu>1 \ \text{or} \
% N=1, \vspace{.1cm}
% \\
%p_G(N,\mu)=1+\frac2{N-1+\mu}
%&
% \ \text{for} \ \mu\le 1\   \ \text{and}\  N\ge 2.
%\end{array}
%\right.
%\end{equation}
%We expect that the critical value for $p$ obtained in Theorem \ref{th1} is
% a serious candidate to the critical exponent which characterizes the threshold between the blow-up and the global existence regions,
%  for $\mu$ small. Obviosly, a rigorous confirmation should be carried out through global existence results.
%\end{rem}

\begin{rem}
Note that the result in
Theorem  \ref{th1} holds true after replacing the linear damping term in \eqref{A}  $V(x) \partial_tu$ by $b(x)\partial_tu$ with $\big(b(x)-V(x)\big)$ belongs to $L^1(\R^N)$. The proof of this  generalized damping case can be obtained by following the same steps as in the proofs of Theorem \ref{th1}  with the necessary modifications related the test functions.
\end{rem}

\begin{rem}
The techniques used in this article can be easily adapted in other
contexts. More precisely, the case of the equation  \eqref{A} with mass term $\frac{\nu^2}{1+|x|^2}u$, for  suitable  values of $\mu$ and $\nu$. Furthermore, we can use the aforementioned
techniques to study  the associated  system of \eqref{A}.
\end{rem}
%
% In addition, we
%the blow up result  stated in Corollary \ref{cor1}
%   shows that there exist a   blow-up region not depending  on the  parameter $\mu$. Therefore
% we deduce that, in the case $\mu>1$,    the result obtained in Corollary \ref{cor1}
%% with critical damping term of time-dependent type \eqref{T-sys1}
% is an  amelioration  to the result  written in Theorem \ref{th1}.
%   the blow-up interval, $p \in (1, p_G(1+\mu)]$ obtained recently  in \cite{Our2}.
%
Throughout this article, we will denote by $C$  a generic positive constant which may depend on the data ($p,\mu,N,f,g$) but not on $\varepsilon$ and whose the value may change from line to line. Nevertheless, we will precise the dependence of the constant $C$ on the parameters of the problem when it is necessary.\\

\section{ Blow-up results in higher dimensions}
%%%%%%%%%%%%%%%%%%%%%%%%%%%%
%%%%%%%%%%%%%%%%%%%%%%%%%%%%
This section is devoted to the proof of Theorem \ref{th1} which is somehow related to
determine the critical exponent associated with the nonlinear term in the problem
\eqref{A} in the higher dimentional  space.
%\subsection{Some auxiliary results}\label{aux}

\par

It is well known that the choice of the test function that will be introduced later  is crucial.
% Naturally, in terms of
%dynamics of the solution of \eqref{A}, the more accurate is the choice of the test function the
%better lifespan estimate is obtained.
In  fact, we construct  a particular positive  solution
 $\psi(x, t)$  with separated variables and satisfies the conjugate equation corresponding to  the linear problem, namely $\psi(x, t)$  satisfies
\begin{equation}\label{lambda-eq}
\partial^2_t \psi(x, t)-\Delta \psi(x, t) -V(x)\partial_t\psi(x, t)=0.
\end{equation}
More precisely, we choose the function $\psi$ given by:
\begin{equation}
\label{psi33}
\psi(x,t):=\rho(t)\phi(x);
\quad
\rho(t):=e^{-t}
\end{equation}
where $\phi(x)$  is  a solution of the elliptic problem
\begin{equation}\label{phi}
\Delta \phi(x) =\left(1+V(x)\right)\phi(x), \qquad \forall  x\in \R^N.
\end{equation}
Note that  the existence  of a  positive solution for the elliptic problem \eqref{phi} is studied    in \cite{LLTW} in the case where $N\ge 2$.
In fact, from Lemma 2.4  in \cite{LLTW}, we know that  there exists a    $ \mathcal{C}^{2}( \R^N)$
%First of all, we shall recall that the
 function $\phi$    solution of \eqref{phi}
which satisfies
\begin{equation}
\label{elip1}
0<\phi (x)\le C_0 (1+|x|)^{-\frac{N-1-\mu}{2}}e^{|x|}, \qquad \forall x\in \R^N,
\end{equation}
for some $C_0>0$.\\

\par

Now, we are in a position to state and prove  the following:
\begin{lem}
\label{lem1}
There exists a constant $C=C(N,R,\mu)>0$ such that
\begin{equation}
\label{psi}
\int_{|x|\leq R+t}\psi(x,t)dx
\leq C
 (1+t)^{\frac{N-1+\mu }{2}},
\quad\forall \ t\ge0.
\end{equation}
\end{lem}
\begin{proof}
Integrating  \eqref{elip1} over the set $\{x, |x|\le R+t\}$, implies
\begin{equation}
\label{psi2}
\int_{|x|\leq R+t}\phi(x)dx
\leq C\int_0^{ R+t}r^{N-1}
 (1+r)^{-\frac{N-1-\mu}{2}}e^{r}
dr,
\quad\forall \ t\ge0.
\end{equation}
In addition, it is easy to get
\begin{equation}
\label{psi3}
\int_0^{ R+t}r^{N-1}
 (1+r)^{-\frac{N-1-\mu}{2}}e^{r}dr
\le  C
 (1+t)^{\frac{N-1+\mu }{2}}e^{t},
\quad\forall \ t\ge0.
\end{equation}
Thus, combining \eqref{psi2} with \eqref{psi3}, we derive that
\begin{equation}
\label{psi4}
\int_{|x|\leq R+t}\phi(x)dx
\le  C
 (1+t)^{\frac{N-1+\mu }{2}}e^{t},
\quad\forall \ t\ge0.
\end{equation}
Employing the estimate \eqref{psi4} and  the expression of $\psi$ given by \eqref{psi33}, we deduce \eqref{psi}.
This concludes the proof of  Lemma \ref{lem1}.
\end{proof}
Now, we are ready to give the proof of  Theorem \ref{th1}.
%Rk: we will use just the paaerticular case where $q=1$ probably, its should be better to write just the case $q=1$?!
%\subsection{Proof of Theorem\til\ref{th1}}
\begin{proof}
[Proof of Theorem\til\ref{th1}]
For the strategy of proof, we basically follow the test function method.

Let $\eta$  be a non-increasing cut-off function such that
 $\eta(r)\in \mathcal{C}^{\infty}([0, +\infty))$  and satisfies
\begin{equation}\label{102}
\eta(r) :=
 \left\{
 \begin{aligned}
 &1
 && \text{for $r\le\tfrac12$},
 \\
 &\text{decreasing}
 && \text{for $\tfrac12<r<1$},
 \\
 &0
 && \text{for $r\ge 1$}.
 \end{aligned}
 \right.
\end{equation}
Let $T>0$. Now, we introduce the following test function:
\begin{equation}\label{phi0}
\Phi( x,t):=
 \left\{
 \begin{aligned}
 &
 && - \partial_t\left(\eta_M^{2p'}(t)\psi(x,t)\right)  \chi(x,t),&\text{for $t>0$},
 \\
 &
 && \phi(x) \chi(x,0),&  \text{for $t=0$},
 \end{aligned}
 \right.
\end{equation}
where $M\in (1, T)$, $p'=\frac{p}{p-1}$, the function $\psi$ is given by \eqref{psi33} and

\begin{equation}\label{103}
\eta_M(t) := \eta\left(\frac tM \right),
\qquad
\chi(x,t) := \eta\left( \frac{|x|}{2 (R+t)} \right).
\end{equation}
Using Definition \ref{def1} and performing an integration by parts in
space in the second term in the right-hand side of \eqref{weaksol}, we write
\begin{equation}\label{weaksol5}
\begin{split}
& \e \int_{\R^N}g(x)\varphi( x,0)dx
+\int_0^T\int_{\R^N}|\partial_tu(x,t)|^p\varphi(x, t) \, dxdt \\
=&- \int_0^T\int_{\R^N} \partial_tu( x,t)\partial_t\varphi( x,t) \, dxdt
- \int_0^T\int_{\R^N}
\  u(x, t)\Delta\varphi(x, t) \, dxdt \\
&+ \int_0^T\int_{\R^N} V(x) \partial_tu(x,t)\varphi( x,t) \, dxdt.
\end{split}
\end{equation}
It is worth mentioning that,
 $\Phi(x,t) \in \mathcal{C}^1_0([0,+\infty)\times\R^N) \cap \mathcal{C}^{\infty}((0,+\infty)\times\R^N)  $.
Now, substituting in \eqref{weaksol5}
$\varphi(x,t)$ by $\Phi(x,t)$, exploiting the compact
 support condition \eqref{suppcond} on $u$, we get
\begin{align}\label{weaksol6}
 \e \int_{\R^N}g(x)\phi( x)dx
-\int_0^T\int_{\R^N}|\partial_tu(x,t)|^p\partial_t\left(\eta_M^{2p'}(t)\psi(x,t)\right)   \, dxdt
\qquad \qquad\qquad\qquad\qquad\qquad\qquad\qquad
 \nonumber\\
= \int_0^T\int_{\R^N}  \partial_tu( x,t) \partial^2_t\left(\eta_M^{2p'}(t)\psi(x,t)\right)   \, dxdt
+ \int_0^T\int_{\R^N}
\  u(x, t) \partial_t\left(\eta_M^{2p'}(t)\Delta \psi(x,t)\right)  \, dxdt\qquad\qquad\qquad\nonumber \\
- \int_0^T\int_{\R^N} V(x) \partial_tu(x,t) \partial_t\left(\eta_M^{2p'}(t)\psi(x,t)\right)   \, dxdt.\qquad\qquad\qquad\qquad\qquad\qquad\qquad
\end{align}
Now, using the fact $\partial_t \psi=-\psi$, we can write
\begin{equation}\label{testfunc1}
-\partial_t \big(\eta_M^{2p'}(t) \psi(x,t)\big)=  \eta_M^{2p'}(t) \psi(x,t) %\chi(x,t)
-\partial_t \big(\eta_M^{2p'}(t) \big)\psi(x,t), %\chi(x,t),
 \qquad \forall t\ge0.
\end{equation}
Therefore, by exploiting the compact
 support condition \eqref{suppcond} on $u$,  integrating by parts, \eqref{testfunc1} and  \eqref{weaksol6}, we deduce that
\begin{align}\label{weaksol7}
 &\e \int_{\R^N}\big(\Delta f(x)+g(x)\big)\phi( x)dx+\int_0^T\int_{\R^N}|\partial_tu(x,t)|^p\eta_M^{2p'}(t)\psi(x,t)   \, dxdt \nonumber\\
&-\int_0^T\int_{\R^N}|\partial_tu(x,t)|^p\partial_t\left(\eta_M^{2p'}(t)\right)\psi(x,t)   \, dxdt \\
=& \int_0^T\int_{\R^N}  \partial_tu( x,t) \psi(x,t)\Big(\partial^2_t\left(\eta_M^{2p'}(t)\right)-2 \partial_t\left(\eta_M^{2p'}(t)\right)\Big)  \, dxdt\nonumber\\
&+ \int_0^T\int_{\R^N}
\ \partial_tu(x, t)\eta_M^{2p'}(t)
\Big(\partial^2_t\psi(x,t)- \Delta \psi(x,t)- V(x) \partial_t\psi(x,t)\Big)
  \, dxdt\nonumber \\
&- \int_0^T\int_{\R^N} V(x) \partial_tu(x,t) \psi(x,t) \partial_t\left(\eta_M^{2p'}(t)\right)  \, dxdt.\nonumber
\end{align}
By exploiting the fact  $\partial_t\left(\eta_M^{2p'}(t)\right)\le 0$,
  and taking into account that $\psi$ verify \eqref{lambda-eq}, we obtain
\begin{align}\label{weaksol8}
 \e \  C_0+\int_0^T\int_{\R^N}|\partial_tu(x,t)|^p\eta_M^{2p'}(t)\psi(x,t)   \, dxdt
 %\nonumber\\
\le & \int_0^T\int_{\R^N}  \partial_tu( x,t) \psi(x,t)\partial^2_t\left(\eta_M^{2p'}(t)\right) \, dxdt\nonumber\\
 &-2 \int_0^T\int_{\R^N}  \partial_tu( x,t) \psi(x,t) \partial_t\left(\eta_M^{2p'}(t)\right)\Big)  \, dxdt\nonumber\\
&- \int_0^T\int_{\R^N} V(x) \partial_tu(x,t)\psi(x,t)  \partial_t\left(\eta_M^{2p'}(t)\right)  \, dxdt\nonumber\\
&=:I_1+I_2+I_3,
\end{align}
where
$$
C_0 \equiv C_0(f,g) := \intspc (\Delta f(x)+ g(x))\phi(x) dx >0
$$
is a positive constant thanks to \eqref{C0}.

\par

 Now, let us  define the functions
\begin{equation}\label{101}
\theta(t) :=
\left\{
\begin{aligned}
&0
&& \text{for $t\le\tfrac12$,}
\\
&\eta(t) && \text{for $t>\tfrac12$,}
\end{aligned}
\right.
\end{equation}
and
\begin{equation}\label{thetaM}
\theta_{M}(t):=\theta\left(\frac tM \right).
\end{equation}
A straightforward computation implies the following inequalities:
\begin{align}
\Big|\partial_t \left(\eta_M^{2p'}(t)\right)\Big|
	&\le  \frac{C}{M}
	\theta_M^{\frac{2p'}{p}}(t),\label{E1}
	\\
\Big|\partial^2_t \left(\eta_M^{2p'}(t)\right)\Big|	&\le
	\frac{C}{M^2}
	\theta_M^{\frac{2p'}{p}}(t).\label{E2}
\end{align}
Again here thanks to the fact that
 the compact
 support condition \eqref{suppcond} on $u$, \eqref{psi}, \eqref{E2}, and H\"{o}lder's inequality, we deduce that

\begin{align}\label{D2}
I_1\le& \, \frac{C}{M^2}
\left(\int_{\frac M2}^M\int_{|x|\le t+R}\ \psi(x,t)  \, dxdt \right)^{\frac{1}{p'}}
\left(\int_0^T\int_{\R^N}|\partial_tu(x,t)|^p\ \theta_M^{2p'}(t) \psi(x,t)  \, dxdt \right)^{\frac1p}\nonumber\\
\le& \, \frac{C}{M^2}
\left( \int_{\frac M2}^M
 (1+t)^{\frac{N-1+\mu }{2}}dt \right)^{\frac{1}{p'}}
 \left(\int_0^T\int_{\R^N}|\partial_tu(x,t)|^p\ \theta_M^{2p'}(t) \psi(x,t)  \, dxdt \right)^{\frac1p}\\
\le& \, CM^{-1+\frac{(N-1+\mu)(p-1)-2}{2p}}
 \left(\int_0^T\int_{\R^N}|\partial_tu(x,t)|^p\ \theta_M^{2p'}(t) \psi(x,t)  \, dxdt \right)^{\frac1p}.\nonumber
\end{align}
Similarly, by \eqref{E2} we obtain
\begin{align}\label{D3}
I_2
\le& \, CM^{\frac{(N-1+\mu)(p-1)-2}{2p}}
 \left(\int_0^T\int_{\R^N}|\partial_tu(x,t)|^p\ \theta_M^{2p'}(t) \psi(x,t)  \, dxdt \right)^{\frac1p}.
\end{align}
Also, similar estimations yield to
\begin{align}\label{D4}
I_3
\le& \, CM^{\frac{(N-1+\mu)(p-1)-2}{2p}}
 \left(\int_0^T\int_{\R^N}|\partial_tu(x,t)|^p\ \theta_M^{2p'}(t) \psi(x,t)  \, dxdt \right)^{\frac1p}.
\end{align}
Gathering \eqref{weaksol8}, \eqref{D2}, \eqref{D3} and \eqref{D4}, we infer
\begin{align}\label{f1}
 \e   C_0+\int_0^T\int_{\R^N}|\partial_tu(x,t)|^p\eta_M^{2p'}(t)\psi(x,t)   \, dxdt
 &\le CM^{\frac{(N-1+\mu)(p-1)-2}{2p}}\\
 &\times \left(\int_0^T\int_{\R^N}|\partial_tu(x,t)|^p\ \theta_M^{2p'}(t) \psi(x,t)  \, dxdt \right)^{\frac1p}.\nonumber
\end{align}
Now, we introduce the following:
\begin{equation}\label{ff}
F(M)=\int_1^M\left(\int_0^T\int_{\R^N}|\partial_tu(x,t)|^p\psi(x,t)\theta_{\rho}^{2p'}(t) \, dxdt \right)\frac1{\rho}d\rho,\qquad  \forall M\in [1,T).
 \end{equation}
 Utilizing  the definition of  $\theta_M$
(given by \eqref{thetaM}), we easily write
%A straightforward calculation gives
\begin{align}\label{sigma0}
F(M)
=&\int_0^T\int_{\R^N}|\partial_tu(x,t)|^p\psi(x,t)\int_1^M\theta^{2p'}(\frac{t}{\rho})\frac1{\rho}d\rho  \, dxdt \nonumber\\
=&\int_0^T\int_{\R^N}|\partial_tu(x,t)|^p\psi(x,t)\int_{\frac tM}^t\theta^{2p'}(\rho)\frac1{\rho}d\rho  \, dxdt.
\end{align}
Let us recall from the expressions of $\theta$, $\theta_M$, $\eta$ and $\eta_M$  defined  in \eqref{101}, \eqref{thetaM}, \eqref{102} and
\eqref{103} that we have
\begin{align}\label{sigma}
F(M)
\le&\int_{0}^T\int_{\R^N}|\partial_tu(x,t)|^p\psi(x,t)\eta^{2p'}\left(\frac tM\right)\int_{\frac 12}^1 \frac1{\rho}d\rho  \, dxdt\nonumber \\
=& \ln2\int_0^T\int_{\R^N} |\partial_tu(x,t)|^p \psi(x,t)  \eta_M^{2p'}(t)\, dxdt.
\end{align}
A differentiation in $M$ of the equation \eqref{ff} gives
\begin{equation}\label{diff}
F'(M)=\frac1{M}\int_0^T\int_{\R^N}|\partial_tu(x,t)|^p\psi(x,t)\theta_M^{2p'}(t) \, dxdt,\qquad  \forall M\in [1,T).
 \end{equation}
Combining \eqref{f1} together with \eqref{sigma} and \eqref{diff},  we get
\begin{equation*}\label{IneqforY}
M^{\frac{(N-1+\mu)(p-1)}{2}} F'(M) \ge C \left(C_0\e+  F(M)\right)^p,\qquad  \forall M\in [1,T).
\end{equation*}
Therefore, we easily obtain the blowup in
finite time for the functional
$F(M)$.
This follows \eqref{T-epss}  and we complete the proof of Theorem \ref{th1}.

\end{proof}

\section{ Blow-up results in the case of weak solutions}

%\subsection{Proof of Theorem\til\ref{th2} and Theorem\til\ref{th3} }

In this section, %we consider the one space dimensional case (N = 1).
we prove Theorem \ref{th2}
 here.

\subsection{Proof of Theorem\til\ref{th2}  }

\begin{proof}[Proof of Theorem\til\ref{th2}]

For the strategy of proof, we basically follow the test function method.

\par

Let $\xi$ be a cut-off  function such that
 $\xi(r)\in \mathcal{C}^{\infty}([0, +\infty))$, $0\leq\xi\leq1$,  and satisfies
\begin{equation}\label{102A}
\xi(r):=
 \left\{
 \begin{aligned}
 &0 && \text{for $0\leq r\leq1$},
 \\
  &\text{increasing}
 && \text{for $1\leq r\leq2$},
 \\
 &1 && \text{for $2\leq r\leq3$},
 \\
 &\text{decreasing}
 && \text{for $3\leq r\leq4$},
 \\
 &0
 && \text{for $r\geq 4$}.
 \end{aligned}
 \right.
\end{equation}
Let $\eta$ be a cut-off function such that
 $\eta(r)\in \mathcal{C}^{\infty}([0, +\infty))$  and satisfies
\begin{equation}\label{102AA}
\eta(r) :=
 \left\{
 \begin{aligned}
 &1
 && \text{for $r\le\tfrac12$},
 \\
 &\text{decreasing}
 && \text{for $\tfrac12<r<1$},
 \\
 &0
 && \text{for $r\ge 1$}.
 \end{aligned}
 \right.
\end{equation}
If $T_{\varepsilon}\le 1$, then the assertion is trivial by choosing $\e$  small enough.
 Assume  that $T_{\varepsilon}\ge 1$  and let $T\in (1,T_{\varepsilon})$.
 Now, we introduce the following test function:
\begin{equation}\label{TestA}
\Phi( x,t):=\eta_{T}^{k}(t)\phi_T^\ell(x),
 \end{equation}
where $k,\ell\geq 2p^{\prime}$, and
\begin{equation}\label{103A}
\eta_{T}(t) := \eta\left(\frac{t}{T} \right),
\qquad
\phi_T(x) := \xi\left( \frac{|x|}{T} \right).
\end{equation}
Let us define an additional function $\zeta_{T}=\zeta_{T}(t)$ such that
\begin{align}\label{zeta}
\zeta_{T}(t):=\int\nolimits_t^{\infty}\eta^k_{T}(\tau)\,\mathrm{d}\tau.
\end{align}
From \eqref{zeta}, we write $\zeta_{T}^\prime(t)=-\eta^k_{T}(t)$, and $\hbox{supp }\zeta_{T}\subseteq[0,{T}]$.

Substituting in  \eqref{weaksolbis}
$\varphi(x,t)$ by $\Phi(x,t)$, %making  use the compact support condition \eqref{suppcond} on $u$,
 % and performing an integration by parts for the second  term in the second line,
  we get
\begin{align}\label{weaksol6A}
& \e \int_{\R^N}g(x)\phi_T^\ell( x)dx
+\int_0^{T} \int_{\R^N}|\partial_tu(x,t)|^p\Phi(x,t)   \, dxdt \\
=& -k \int_0^{T} \int_{\R^N} \partial_tu( x,t)\eta_{T}^{k-1}(t)\eta^\prime_{T}(t)\phi_T^\ell(x) \, dxdt
+\int_0^{T} \int_{\R^N}\frac{d}{dt}\zeta_{T}(t)
   u(x, t)
 \Delta\big( \phi_T^\ell(x)\big) \, dxdt \nonumber\\
&+ \int_0^{T} \int_{\R^N} V(x) \partial_tu(x,t) \, \eta_{T}^{k}(t)\phi_T^\ell(x)dxdt.\nonumber
\end{align}
Performing an integration by parts for the second  term in the second line yields
\begin{align}\label{weaksol7A}
& \e \int_{\R^N}g(x)\phi_T^\ell( x)dx
+\int_0^{T} \int_{\R^N}|\partial_tu(x,t)|^p\Phi(x,t)   \, dxdt \\
=& -k \int_0^{T} \int_{\R^N} \partial_tu( x,t)\eta_{T}^{k-1}(t)\eta^\prime_{T}(t)\phi_T^\ell(x) \, dxdt
-\int_0^{T} \int_{\R^N}
\   \partial_tu(x, t)\zeta_{T}(t) \Delta\big( \phi_T^\ell(x)\big)\, dxdt \nonumber\\
&+ \int_0^{T} \int_{\R^N} V(x) \partial_tu(x,t) \, \eta_{T}^{k}(t)\phi_T^\ell(x)dxdt- \e \zeta_{T}(0) \int_{\R^N}
\  f(x) \Delta\big( \phi_T^\ell(x)\big)\, dx.\nonumber
\end{align}
%Let
%$$I_{T}:=\int_0^{T} \int_{\R^N}|u_t(x,t)|^p\Phi(x,t)\,dxdt,$$
%we obtain
By using the identity \eqref{weaksol7A},
 we get that
\begin{equation}\label{weaksol8A}
 \e \zeta_{T}(0) \int_{\R^N}
\  f(x) \Delta\big( \phi_T^\ell(x)\big)\, dx+\e \int_{\R^N}g(x)\phi_T^\ell( x)dx
%+
%\e \zeta_{T}(0) \int_{\R^N}
%\  f(x) \Delta\big( \phi_T^\ell(x)\big)\, dx
+\int_0^{T} \int_{\R^N}|\partial_tu(x,t)|^p\Phi(x,t)   \, dxdt\le J_1+J_2+J_3,
%+J_4,
%
%k \int_0^{T} \int_{\R^N}|u_t( x,t)|\eta_{T}^{k-1}(t)|\eta^\prime_{T}(t)|\phi^\ell(x) \, dxdt\nonumber\\
% &+ \int_0^{T} \int_{\R^N}\  |u_t(x, t)|\Psi_{T}(t)|\Delta\phi^\ell(x)| \, dxdt\nonumber\\
%&+ \int_0^{T} \int_{\R^N} V(x) |u_t(x,t)| \, \eta_{T}^{k}(t)\phi^\ell(x)dxdt\nonumber\\
%&+ \int_{\R^N}\ |f(x)|\zeta_{T}(0)|\Delta\phi^\ell(x)| \, dx\nonumber\\
%&=:I_1+I_2+I_3+I_4.
\end{equation}
where
\begin{align*}\label{defJ1234}
J_1=&:k \int_0^{T} \int_{\R^N}|\partial_tu( x,t)|\eta_{T}^{k-1}(t)|\eta^\prime_{T}(t)|\phi_T^\ell(x) \, dxdt,\nonumber\\
J_2=&: \int_0^{T} \int_{\R^N}\  |\partial_tu(x, t)|\zeta_{T}(t)| \Delta\big( \phi_T^\ell(x)\big)| \, dxdt,\nonumber\\
J_3=&: \int_0^{T} \int_{\R^N} V(x) |\partial_tu(x,t)| \, \eta_{T}^{k}(t)\phi_T^\ell(x)dxdt,\nonumber\\
%J_4=&: \e \zeta_{T}(0) \int_{\R^N}\ |f(x)|| \Delta\big( \phi_T^\ell(x)\big)| \, dx.
\end{align*}
Let $ \nu>0$. By applying $\nu$-Young's inequality
\begin{equation}\label{Young}
AB\leq\nu A^p+C(\nu,p)B^{p^\prime},\,\, A\geq0,\;B\geq0,\;p+p^\prime=pp^\prime,\,\, C(\nu,p)=(\nu\,p^p)^{-1/(p-1)}(p-1),
\end{equation}
we get
\begin{equation}\label{D2A}
J_1\  \leq\    \nu \int_0^{T} \int_{\R^N}|\partial_tu(x,t)|^p\Phi(x,t)     \, dxdt
+C(\nu)\int_0^{T} \int_{\R^N}\eta_{T}^{k-p^\prime}(t)|\eta^\prime_{T}(t)|^{p^\prime}\phi_T^\ell(x) \, dxdt.
\end{equation}
By \eqref{D2A}, the fact that $k\ge p^{\prime}$ and taking into account the expression of  $\eta_{T}$ and $\phi_T$ given by  \eqref{103A}, we
conclude
\begin{equation}\label{D2Abis}
J_1
\le \nu \int_0^{T} \int_{\R^N}|\partial_tu(x,t)|^p\Phi(x,t)     \, dxdt+C(\nu){T}^{N+1-{p^\prime}}.
\end{equation}
Similarly, we obtain
\begin{equation}\label{D3A}
J_2\le \nu \int_0^{T} \int_{\R^N}|\partial_tu(x,t)|^p\Phi(x,t)     \, dxdt
+C(\nu)\, \int_0^{T} \int_{\R^N}\eta_{T}^{-k\frac{p^\prime}{p}}(t)\zeta^{p^\prime}_{T}(t)\phi_T^{-\ell\frac{p^\prime}{p}}(x)
| \Delta\big( \phi_T^\ell(x)\big)| ^{p^\prime} \, dxdt.
\end{equation}
Using the fact that $\eta_T$ is decreasing and $\hbox{supp }\eta_{T}\subseteq[0,{T}]$, we obtain
\begin{equation}\label{D3AA}
\eta_{T}^{-k\frac{p^\prime}{p}}(t)\zeta^{p^\prime}_{T}(t)\leqslant \eta_{T}^{-k\frac{p^\prime}{p}}(t) \left(\int\nolimits_t^{T}\eta^{k}_{T}(\tau)\,\mathrm{d}\tau\right)^{p^\prime}\leqslant
 {T}^{p^\prime} \eta^{k}_{T}(t)\leqslant  T^{p^\prime}.
\end{equation}
In addition,  by exploiting the identity $
 \Delta
\big( \phi_T^\ell(x)\big)=\ell\phi_T^{\ell-1}(x)
 |\Delta \phi_T(x)+\ell(\ell-1)\phi_T^{\ell-2}(x)|\nabla \phi_T(x)|^2,$  we deduce
\begin{align*}
\phi_T^{-\ell\frac{p^\prime}{p}}(x)| \Delta\phi_T^\ell(x)|^{p^\prime}
\leqslant  C\phi_T^{\ell -p^{\prime}}(x)|\Delta \phi_T(x)|^{p^\prime} +C\phi_T^{\ell -2p^{\prime}}(x)|\nabla \phi_T(x)|^{2p^\prime} \leqslant  C\phi_T^{\ell -2p^{\prime}}(x)T^{-2p^\prime}.
\end{align*}
Plugging the above inequality, \eqref{D3AA} and the fact that $\ell-2p^{\prime}\ge0$ into \eqref{D3A}, we get
\begin{equation}\label{D2Abis1}
J_2
\le \nu \int_0^{T} \int_{\R^N}|\partial_tu(x,t)|^p\Phi(x,t)     \, dxdt+C(\nu){T}^{N+1-{p^\prime}}.
\end{equation}
In the same way, thanks to \eqref{Young}, we infer
\begin{equation}\label{D2Abis2}
J_3
\le \nu \int_0^{T} \int_{\R^N}|\partial_tu(x,t)|^p\Phi(x,t)     \, dxdt+C(\nu)\int_0^{T} \int_{\R^N}\eta_{T}^{k}(t)\phi_T^\ell(x)|V(x)|^{p^\prime}\, dxdt.
\end{equation}
To estimate the second term on the right-hand side, we have
\begin{equation}\label{D2Abis2hhh}
 \int_{\R^N}\phi_T^\ell(x)|V(x)|^{p^\prime}\, dx=\int_{T\leq|x|\leq 4T}\phi_T^\ell(x)|V(x)|^{p^\prime}\, dx\leq CT^{-p^\prime} \int_{|x|\leq 4T}\phi^{\ell}(x)\, dx\leq C\,T^{N-p^\prime}.
\end{equation}
Consequently, we derive
\begin{equation}\label{D2Abis111}
J_3
\le \nu \int_0^{T} \int_{\R^N}|\partial_tu(x,t)|^p\Phi(x,t)     \, dxdt+C(\nu){T}^{N+1-{p^\prime}}.
\end{equation}
Gathering \eqref{weaksol8A}, \eqref{D2Abis}, \eqref{D2Abis1}, and  \eqref{D2Abis111} and choosing $\nu$ small enough,
we deduce
\begin{align}\label{Estimate1}
\e \zeta_{T}(0) \int_{\R^N}
\  f(x) \Delta\big( \phi_T^\ell(x)\big)\, dx+ \e \int_{\R^N}g(x)\phi_T^\ell( x)dx
+\int_0^{T} \int_{\R^N}|\partial_tu(x,t)|^p\Phi(x,t)   \, dxdt
\le
C{T}^{N+1-{p^\prime}}.
\end{align}

Now, we distinguish three cases:\\

\noindent {\underline{Case I:}}
Let us denote in this case  $f\equiv0$ and $g$ satisfies
\begin{equation}\label{case1}
g(x)\geq\frac{\Pi_1(|x|)}{(1+|x|)^\kappa},\end{equation}
%$$%\quad\hbox{or}\quad \lim_{|x|\to \infty}|x|^{\kappa}g(x)= L,$$
  where $\Pi_1\equiv C$ if $0<\kappa<\kappa_1:=\frac{1}{p-1}$, and $\Pi_1$ is positive, monotonously increasing, $\d{\lim_{r\rightarrow\infty}\Pi_1(r)=\infty}$ if $\kappa=\kappa_1$. 
In this case, the inequality \eqref{Estimate1} becomes
\begin{align}\label{Estimate1av1}
 \e \int_{\R^N}g(x)\phi_T^\ell( x)dx
+\int_0^{T} \int_{\R^N}|\partial_tu(x,t)|^p\Phi(x,t)   \, dxdt
\le
C{T}^{N+1-{p^\prime}}.
\end{align}
%$\bullet$ If $\displaystyle \lim_{|x|\to \infty}|x|^{\kappa}g(x)= L$, there exist $T_0>1$ large enough such that,  we have
%$$
% \e \int_{\R^N}g(x)\phi_T^\ell( x)dx\ge  \e \int_{2T\leq|x|\leq 3T}g(x)dx\ge
% C\e \int_{2T\leq|x|\leq 3T}\frac{1}{|x|^{\kappa}} dx\geq C\e\,T^{N-\kappa},
%$$
%for any $T>T_0$.  Hence
%$$
% C_0\e T^{N-\kappa } \le
%C_1{T}^{N+1-{p^\prime}},\qquad  \forall\, T>T_0,
%$$
%i.e.
% $$
% C_0\e \le
%2C_1{T}^{\kappa+1-{p^\prime}}, \qquad  \forall \,T>T_0.
%$$
%   which leads, using $\kappa<p^\prime-1$, to a contradiction by letting $T\rightarrow\infty$. Obviously, in this case it is easy to derive
%$$
% T_{\e}\le
% C\e^{-\frac{p-1}
% {1-\kappa(p-1)}},\quad\hbox{ for all}\,\,\varepsilon\leq\varepsilon_0,
% $$
% for some $\varepsilon_0>0$. Indeed, if $T\leq T_0$, then by choosing $\varepsilon_0>0$ such that $C\varepsilon_0^{-\frac{p-1} {1-\kappa(p-1)}}=T_0$, we may conclude that for all $\varepsilon\leq\varepsilon_0$, we have $T\leq C\e^{-\frac{p-1}
% {1-\kappa(p-1)}}$.\\
%$\bullet$ If $g(x)\geq\frac{\phi(|x|)}{(1+|x|)^\kappa}$, we have
By exploiting  \eqref{102A}, \eqref{103A} and \eqref{case1},  we conclude that  we have % there exist $T_0>1$ large enough such that,  we have
\begin{equation}\label{160410}
 \e \int_{\R^N}g(x)\phi_T^\ell( x)dx\ge  \e \int_{2T\leq|x|\leq 3T}g(x)dx\ge
 \e \int_{2T\leq|x|\leq 3T}\frac{\Pi_1(|x|)}{(1+|x|)^{\kappa}} dx\geq C\Pi_1(2T)\e\,T^{N-\kappa},\end{equation}
for any $T>1$.  By combining \eqref{Estimate1av1} and \eqref{160410}, we obtain
%$$
% C_0\phi(2T)\e T^{N-\kappa } \le
%C_1{T}^{N+1-{p^\prime}},\qquad  \forall\, T>1,
%$$
%i.e.
 $$
 C\Pi_1(2T)\e \le
{T}^{\kappa+1-{p^\prime}}, \qquad  \hbox{for all}\,\,T>1,
$$
  which leads, using $\kappa\leq p^\prime-1$, to a contradiction by letting $T\rightarrow\infty$. In addition, when $\kappa< p^\prime-1$, it is easy to derive
  there exists a constant $\varepsilon_0= \varepsilon_0(g,N, p, \mu)> 0$ such
that $T_{w}$ satisfies
$$
 T_{w}\le
 C\e^{-\frac{p-1}
 {1-\kappa(p-1)}},\quad\hbox{ for all}\,\,\varepsilon\leq\varepsilon_0.
 $$

\noindent \underline{Case II:} First, we recall here that $f$ and $g$ satisfy
\begin{equation}\label{DF}
\Delta f(x)\geq\frac{\Pi_1(|x|)}{(1+|x|)^{\kappa+1}}\qquad\hbox{and}\qquad g(x) \geq0,
\end{equation}
where $\Pi_1\equiv C$, if $0<\kappa<\frac{1}{p-1}$, and $\Pi_1$ is positive, monotonously increasing, $\d{\lim_{r\rightarrow\infty}\Pi_1(r)=\infty}$ if $\kappa=\frac1{p-1}$. 
In this case, the inequality \eqref{Estimate1} implies
%$$
%\e \zeta_{T}(0) \int_{\R^N}
%\  \Delta f(x) \phi_T^\ell(x)\, dx+ \e \int_{\R^N}g(x)\phi_T^\ell( x)dx
%\le
%C{T}^{N+1-{p^\prime}}.
%$$
%$\bullet$ If $g\geq0$, then we have
\begin{equation}\label{AA16}
\e \zeta_{T}(0) \int_{\R^N}\ \Delta f(x) \phi_T^\ell(x)\, dx\le C{T}^{N+1-{p^\prime}}, \qquad \hbox{for all}\,\,  T\ge 1.
\end{equation}
On the other hand, using 
  \eqref{102A}, \eqref{103A} and \eqref{DF}, we get
\begin{equation}\label{DF1}
\e \zeta_{T}(0) \int_{\R^N}\ \Delta f(x) \phi_T^\ell(x)\, dx\ge  
\e \zeta_{T}(0)\int_{2T\leq|x|\leq 3T}\frac{\Pi_1(|x|)}{(1+|x|)^{(\kappa+1)}} dx.
\end{equation}
Therefore,% using \eqref {DF1}and \eqref{DF}, we infer
\begin{equation}
\e \zeta_{T}(0) \int_{\R^N}\ \Delta f(x) \phi_T^\ell(x)\, dx
\geq C\e  \zeta_{T}(0) \Pi_1(2T)\,T^{N-\kappa-1}, \qquad \hbox{for all}\,\,  T\ge 1.
\end{equation}
Furthermore, taking account of
\begin{equation*}\label{10mars1}
\zeta_{T}(0)=\int\nolimits_0^{T}\eta^k\big(\frac{\tau}{T}\big)\,\mathrm{d}\tau\ge \int\nolimits_0^{T/2}\mathrm{d}\tau\ge \frac{T}2,
\end{equation*}
 we infer
 \begin{equation}\label{AA17}
\e \zeta_{T}(0) \int_{\R^N}\ \Delta f(x) \phi_T^\ell(x)\, dx
\geq C\e   \Pi_1(2T)\,T^{N-\kappa}, \qquad \hbox{for all}\,\, T\ge 1.
\end{equation}
 Now, combining (\ref{AA16}) and (\ref{AA17}), we obtain

 \begin{equation}\label{AA18}
 C\e \Pi_1(2T) \le
{T}^{\kappa+1-{p^\prime}}, \qquad  \hbox{for all}\,\,T\ge 1,
\end{equation}
  which leads, using $\kappa\leq p^\prime-1$, to a contradiction by letting $T\rightarrow\infty$. In addition, when $\kappa< p^\prime-1$, it is easy to derive a constant $\varepsilon_0= \varepsilon_0(g,N, p, \mu)> 0$ such
that $T_{w}$ satisfies
$$
 T_{w}\le
 C\e^{-\frac{p-1}
 {1-\kappa(p-1)}},\quad\hbox{ for all}\,\,\varepsilon\leq\varepsilon_0.
 $$
 \\
%
%  which leads, using $\kappa\leq p^\prime-1$, to a contradiction by letting $T\rightarrow\infty$. Obviously, when $\kappa< p^\prime-1$, it is easy to derive
%$$
% T_{\e}\le
% C\e^{-\frac{p-1}
% {1-\kappa(p-1)}},\quad\hbox{ for all}\,\,\varepsilon\leq\varepsilon_0,
% $$
% for some $\varepsilon_0>0$. Indeed, if $T\leq T_0$, then by choosing $\varepsilon_0>0$ such that $C\varepsilon_0^{-\frac{p-1} {1-\kappa(p-1)}}=T_0$, we may conclude that for all $\varepsilon\leq\varepsilon_0$, we have $T\leq C\e^{-\frac{p-1}
% {1-\kappa(p-1)}}$.\\
%$\bullet$ If $g(x)\geq\frac{\phi_2(|x|)}{(1+|x|)^\kappa}$, we have
%$$
% \e \int_{\R^N}g(x)\phi_T^\ell( x)dx\ge  \e \int_{2T\leq|x|\leq 3T}g(x)dx\ge
% C\e \int_{2T\leq|x|\leq 3T}\frac{\phi_2(|x|)}{(1+|x|)^{\kappa}} dx\geq C\phi_2(2T)\e\,T^{N-\kappa},
%$$
%for any $T>1$.  Hence
%$$
% C_0(\phi_1(2T)+\phi_2(2T))\e T^{N-\kappa } \le
%C_1{T}^{N+1-{p^\prime}},\qquad  \forall\, T>1,
%$$
%i.e.
% $$
% C_0(\phi_1(2T)+\phi_2(2T))\e \le
%2C_1{T}^{\kappa+1-{p^\prime}}, \qquad  \forall \,T>1,
%$$
%  which leads, using $\kappa\leq p^\prime-1$, to a contradiction by letting $T\rightarrow\infty$. Obviously, when $\kappa< p^\prime-1$, it is easy to derive
%$$
% T_{\e}\le
% C\e^{-\frac{p-1}
% {1-\kappa(p-1)}},\quad\hbox{ for all}\,\,\varepsilon\leq\varepsilon_0,
% $$
% for some $\varepsilon_0>0$.\\

\noindent \underline{Case III:} Let us recall here $f\in L^1(\R^N)$ and $g$ satisfies
\begin{equation}\label{DF3}
 g(x)\geq\frac{\Pi_1(|x|)}{(1+|x|)^{\kappa}},
\end{equation}
where $\Pi_1\equiv C$, if $0<\kappa<\frac{1}{p-1}$, and $\Pi_1$ is positive, monotonously increasing, $\d{\lim_{r\rightarrow\infty}\Pi_1(r)=\infty}$ if $\kappa=\frac1{p-1}$. 
In this case, the inequality \eqref{Estimate1} implies

\begin{align}\label{1Estimate1}
 \e \int_{\R^N}g(x)\phi_T^\ell( x)dx
%+\int_0^{T} \int_{\R^N}|\partial_tu(x,t)|^p\Phi(x,t)   \, dxdt
\le
C{T}^{N+1-{p^\prime}}+\underbrace{\e \zeta_{T}(0) \int_{\R^N}
\  |f(x)| \big|\Delta\big( \phi_T^\ell(x)\big)\big|\, dx}_{J_4}.
\end{align}
Taking account of  the fact that $f\in L^1(\R)$, the inequality
$\| \Delta\big(\phi_T^\ell(x)\big)\|_{L^{\infty}}\leq CT^{-2}$ and
\begin{equation*}\label{10mars1}
\zeta_{T}(0)=\int\nolimits_0^{T}\eta^k_{T}(\tau)\,\mathrm{d}\tau\le T,
\end{equation*}
 we write
\begin{equation}\label{D5A}
J_4 \leq C\e T^{-1} \int_{\R^N}|f(x)| \, dx= C\|f\|_{L^1(\R)}\e T^{-1}.
\end{equation}
By exploiting  \eqref{102A}, \eqref{103A} and \eqref{DF3},  we conclude that  we have % there exist $T_0>1$ large enough such that,  we have
\begin{equation}\label{16041}
 \e \int_{\R^N}g(x)\phi_T^\ell( x)dx\ge 
 \e \int_{2T\leq|x|\leq 3T}\frac{\Pi_1(|x|)}{(1+|x|)^{\kappa}} dx\geq C\Pi_1(2T)\e\,T^{N-\kappa}, \qquad \hbox{for all}\,\,T\ge 1.
 \end{equation}
  By combining \eqref{1Estimate1}, \eqref{D5A} and \eqref{16041}, we obtain
\begin{align}\label{Estimate22b}
 C_0\e \Pi_1(2T)T^{N-\kappa } \le
C_1{T}^{N+1-{p^\prime}}+C_1\|f\|_{L^1(\R)}\e T^{-1},\qquad \hbox{for all}\,\,T\ge 1,
\end{align}
 which leads, using    $\kappa\le \min (N+1,\frac1{p-1})$, to a contradiction by letting $T\rightarrow\infty$.\\
 
  Moreover,
  when     $\kappa< \min (N+1,\frac1{p-1})$, 
 \eqref{Estimate22b} yields
\begin{align}\label{Estimate22}
 C_0\e T^{N-\kappa} \le
C_1{T}^{N+1-{p^\prime}}+C_1\|f\|_{L^1(\R)}\e T^{-1},\qquad \hbox{for all}\,\, T\ge 1.
\end{align}
Therefore,

$\bullet$ If $T\geq \left(\frac{2C_1\|f\|_1}{C_0}\right)^{1/(N+1-\kappa)}:=T_0$, then 
$$C_1\|f\|_1\varepsilon\,T^{-1}\leq  \frac{C_0\e}{2}\,T^{N-\kappa} \qquad \hbox{for all}\,\, T\ge \max (T_0,1).$$
Hence, the inequality \eqref{Estimate22} becomes
 %$$
 %C_0\e T^{N-\kappa } \le 2C_1{T}^{N+1-{p^\prime}}, \qquad  \forall \,T>T_0.
%$$
%i.e.
 \begin{align}\label{Estimate222}
 C_0\e \le
2C_1{T}^{\kappa+1-{p^\prime}}, \qquad \hbox{for all}\,\,T\ge \max (T_0,1).
\end{align}
   which leads,
   % there exists a constant $\varepsilon_0= \varepsilon_0(f,g,N, p, \mu)> 0$ such
%that 
$T_{w}$ satisfies
$$
 T_{w}\le
 C\e^{-\frac{p-1}
 {1-\kappa(p-1)}},\quad\hbox{ for all}\,\,\varepsilon>0.
 $$
\noindent$\bullet$ If $T\leq \left(\frac{2C_1\|f\|_1}{C_0}\right)^{1/(N+1-\kappa)}$, we may directly conclude that
 $$T\leq  \left(\frac{2C_1\|f\|_1}{C_0}\right)^{1/(N+1-\kappa)}\leq C\e^{-\frac{p-1}
 {1-\kappa(p-1)}},\quad\hbox{ for all}\,\,\varepsilon\leq\varepsilon_0,$$
where
$$\varepsilon_0:=C^{\frac{1-\kappa(p-1)}{p-1}}\left(\frac{C_0}{2C_1\|f\|_1}\right)^{\frac{1-\kappa(p-1)}{(p-1)(N+1-\kappa)}}.$$
This achieves the proof of Theorem \ref{th2}.
 \end{proof}

\bibliographystyle{plain}

%%%%%%%%%%%%%%%%%%%%%%%%%%%%%%%%%%%%%%%%%%%%%%%%%
%%%%%%%%%%%%%%%%%%%%% References %%%%%%%%%%%%%%%%%%%%%
%%%%%%%%
	
\end{document}